\documentclass[a4paper]{amsart}

\usepackage[leqno]{amsmath}
\usepackage{amsthm}
\usepackage{amssymb}
\usepackage{graphicx}
\usepackage{calc} 
\usepackage{txfonts}
\usepackage{enumitem}
\usepackage{ifpdf} 
\usepackage{ifthen}
\usepackage{braket}
\ifpdf
  \usepackage[pdftex]{hyperref}
  \hypersetup{plainpages=false,unicode=true,pdffitwindow=true}
\else
  \usepackage[dvips]{hyperref}
\fi

\newtheorem{theorem}{Theorem}[section]
\newtheorem{lemma}[theorem]{Lemma}
\newtheorem{proposition}[theorem]{Proposition}
\newtheorem{corollary}[theorem]{Corollary}

\newtheorem{example}[theorem]{Example}
\newtheorem{assumption}[theorem]{Assumption}
\newtheorem{remark}[theorem]{Remark}

\newcommand{\R}{\mathbb{R}}
\newcommand{\N}{\mathbb{N}}

\newcommand{\F}{\mathcal{F}}

\ifpdf
\renewcommand{\G}{\mathfrak{g}}
\else
\newcommand{\G}{\mathfrak{g}}
\fi

\newcommand{\f}[2]{\frac{#1}{#2}}
\newcommand{\pa}{\partial}

\newcommand{\dt}{\Delta t}

\newcommand{\barXi}{\overline{\Xi}}

\newcommand{\abs}[1]{\left\lvert#1\right\rvert} 
\newcommand{\norm}[1]{\left\lVert#1\right\rVert} 

\newcommand{\fB}{\mathbf{B}}

\newcommand{\fomega}{\boldsymbol{\omega}}

\newcommand{\Cbv}{C^{\text{1-var}}}

\DeclareMathOperator{\id}{id}

\newcommand{\Prefix}[3]{\vphantom{#3}#1#2#3}

\begin{document}

\title{Cubature on Wiener space: pathwise convergence}

\author[Ch.~Bayer]{Christian Bayer}
\address{WIAS Berlin\\Mohrenstra\ss{}e 39\\10117 Berlin\\Germany}
\email{christian.bayer@wias-berlin.de}

\author[P.~Friz]{Peter K.~Friz}
\address{TU Berlin\\Institute of Mathematics, MA 7-2\\Stra\ss{}e des 17. Juni
  136\\10623 Berlin\\Germany}
\address{WIAS Berlin\\Mohrenstra\ss{}e 39\\10117 Berlin\\Germany}
\email{friz@math.tu-berlin.de, peter.friz@wias-berlin.de}

\thanks{P.~K.~Friz has received funding from the European Research Council
  under the European Union's Seventh Framework Programme (FP7/2007-2013) / ERC
  grant agreement nr.~258237.}

\begin{abstract}
  Cubature on Wiener space [Lyons, T.; Victoir, N.; Proc.~R.~Soc.~Lond.~A 8
  January 2004 vol. 460 no.~2041 169-198] provides a powerful alternative to
  Monte Carlo simulation for the integration of certain functionals on Wiener
  space. More specifically, and in the language of mathematical finance,
  cubature allows for fast computation of European option prices in generic
  diffusion models.

  We give a random walk interpretation of cubature and similar (e.g.~the
  Ninomiya--Victoir) weak approximation schemes. By using rough path analysis,
  we are able to establish weak convergence for general path-dependent
  option prices.
\end{abstract}




\maketitle

\section{Introduction}
\label{sec:introduction}

Cubature on Wiener space (Kusuoka~\cite{kus01,kus04}, Lyons and
Victoir~\cite{lyo/vic04}, see also Litterer and Lyons~\cite{lit/lyo07},
Ninomiya and Victoir~\cite{nin/vic08}, Ninomiya and Ninomiya~\cite{nin/nin09})
provides a powerful alternative to Monte Carlo simulation for the integration
of certain functionals on Wiener space. As of present, these functionals are
of the form $f\left( S_{T}\right) $ where $S_{T}$ is the image of a
$d$-dimensional Brownian motion under the It\^{o}-map (the solution map to a
stochastic differential equation); the aim of cubature on Wiener space is then
to provide a fast numerical algorithm to compute $ E\left[ f\left( S_{T}
  \right) \right] $, where the expectation is taken over the $d$-dimensional
Wiener measure. 

In the language of mathematical finance, cubature deals with European option
prices in generic diffusion models. Although some exotic options can be
handled in this framework\ (e.g. Asian options, by enhancing the state-space)
general path-dependent options are not included in the presently available
analysis on cubature methods. It must be admitted that cubature has been
designed for fast evaluation of payoffs of the type $f\left( S_{T}\right) ;$
but even so, it may maintain its benefits in mildly\ path-dependent situation
and, in any case, convergence to the correct value will be considered a
minimal requirement by most users.

The answer to ``How can it fail to converge to the correct value?'' is not
trivial: cubature methods are essentially derived from replacing
Wiener-measure but a path-space measure supported on smooth paths $\left\{
  \omega _{i}\right\} $, subject to certain technical conditions relating to
the iterated integrals of these paths. Stochastic differential equations,
however, are far from stable under perturbations in the iterated integrals:
recall the well-known examples of McShane~\cite{mcs72} which give uniform
approximations to Brownian motion where the limiting differential equation
exhibits bias in the form of additional drift terms. (The explanation is that
these approximations do \emph{not} correctly approximate the iterated
integrals of Brownian motion known as L\'{e}vy's stochastic area.) At the risk
of confusing the reader, even if is guaranteed that a sample path \emph{and
  its stochastic area} are uniformly correctly approximated, the limiting
differential equation may still exhibit additional drift terms\footnote{
  ... now involving two or more iterated Lie brackets of the diffusion vector
  fields; see ~\cite{fri/obe09} for more on such subtleties.}. The point is
that topology matters: uniform convergence needs to be replaced by a stronger
notion of H\"{o}lder (or $p$-variation) rough path topology in order to use
the stability results of rough path theory.

Our key idea is to view the iterations of cubature steps, (Lyons and
Victoir~\cite[Theorem 3.3]{lyo/vic04} for instance), via an underlying
random walk of the driving signal, Brownian motion plus L\'{e}vy's area. The
iterated cubature scheme corresponds precisely to stochastic differential
equations in which the driving Brownian motion is replaced by $k$ properly
rescaled concatenations of the $\omega _{i}$ (say, chosen independently with
probability $\lambda _{i}$ at each step). Thanks to the smoothness of the
$\omega _{i}$, such a path has canonically defined iterated integrals; the
``only'' thing left to do is to establish weak convergence of this random walk
to Brownian motion and L\'{e}vy's stochastic area, in the correct rough path
topology. It is then an immediate consequence of the continuity of the It\^{o}
map in rough path sense (i.e. as a deterministic function of path and area in
rough path topology) to see that this entails the desired weak convergence
result for path dependent functionals of the type $f\left( S_{t}:0\leq t\leq
  T\right)$.

Weak convergence questions of this type were first discussed in E.~Breuillard,
P.~Friz, M.~Huesmann~\cite{bre/fri/hue09}. Unfortunately, the ``Rough path
Donsker'' theorem obtained therein does not lend itself immediately to the
present applications: a moment of reflection reveals that it would cover
cubature with (1)\ equidistant steps and (2)\ in which the $\omega _{i}$ are
straight lines (Wong-Zakai!). Our strategy is thus to develop refined
arguments that allow to cover the generic cubature setting as well as its
recent variations\ (like Ninomya--Victoir). This leads, en passant, to a more
flexible version of the Donsker theorem for Brownian motion on Lie groups in
topologies considerably finer than the uniform one.

The mathematical content -- weak convergence of discrete structures to
(Stratonovich) SDE solutions -- should also be compared to the (typically
It\^{o}) diffusion limits of Markov chains (cf.~Stroock and
Varadhan~\cite[Section 11.2]{str/var73}), although we shall not pursue this
point further here.

The current paper uses many ideas and results from rough path theory, see
Lyons~\cite{lyo98} and Friz and Victoir~\cite{fri/vic10}, which we primarily
use for reference in this paper. For cubature on Wiener space, the
authoritative reference remains Lyons and Victoir~\cite{lyo/vic04}.

\section{Cubature on Wiener space and the associated random walks}
\label{sec:cubat-wien-space}

Let $B = (B_t)_{t \in [0,1]}$ denote a standard $d$-dimensional Brownian
motion on $\left( \Omega, \F, (\F_t)_{t \in [0,1]}, P \right)$ and $\fB_\cdot
= S_2(B)_{0,\cdot}$, i.e., $\fB$ is the Brownian motion enhanced by its L\'evy
area. The geometrical setting of $\fB$ is the Lie group $G^2(\R^d)$, which can
be defined as follows: let $e_1, \ldots, e_d$ denote the canonical basis of
$\R^d$. Then $e_i \otimes e_j$, $1\le i,j \le d$, forms a basis for the tensor
product $\R^d \otimes \R^d$. Consider the algebra $T^2(\R^d) \coloneqq \R
\oplus \R^d \oplus \R^d \otimes \R^d$, which is understood as a step-$2$
nilpotent non-commutative algebra, i.e., for $\mathbf{x}_1 = z_1+x_1+a_1,
\mathbf{x}_2 = z_2+x_2+a_2 \in T^2(\R^d)$ the product is given by
\begin{equation*}
  \mathbf{x}_1 \otimes \mathbf{x}_2 = z_1 z_2 + (z_2x_1 + z_1x_2) + \left(
    z_2a_1 + z_1a_2 + x_1 \otimes x_2 \right).
\end{equation*}
Consider $\G^2(\R^d) \subset T^2(\R^d)$, the Lie-algebra generated by $e_1,
\ldots, e_d$ together with $[e_i, e_j]$, $1 \le i < j \le d$, with the
commutator defined by $[\mathbf{x}, \mathbf{y}] \coloneqq \mathbf{x} \otimes
\mathbf{y} - \mathbf{y} \otimes \mathbf{x}$, $\mathbf{x}, \mathbf{y} \in
T^2(\R^d)$. The exponential map $\exp: T^2(\R^d) \to T^2(\R^d)$
defined by 
\begin{equation*}
  \exp(\mathbf{x}) \coloneqq 1 + \sum_{k=1}^\infty \f{1}{k!}
  \mathbf{x}^{\otimes k} = 1 + \mathbf{x} + \f{1}{2} \mathbf{x} \otimes
  \mathbf{x} 
\end{equation*}
in the step-$2$ nilpotent setting, maps $\G^2(\R^d)$ in a bijective way to the
Lie group $G^2(\R^d) \coloneqq \exp(\G^2(\R^d)) \subset T^2(\R^d)$.

This Lie group is highly relevant for rough path analysis, since it is the
geometric setting of the enhanced Brownian motion mentioned before. Indeed,
the $T^2(\R^d)$-valued process $\fB$ defined by
\begin{equation}
  \label{eq:enhanced-BM}
  \fB_t \coloneqq 1 + \sum_{i=1}^d B^i_t e_i + \sum_{i,j=1}^d \int_0^t B^i_s
  \circ dB^j_s \, e_i \otimes e_j \eqqcolon S_2(B)_{0,t}, \quad 0 \le t \le 1,
\end{equation}
lives in the Lie group $G^2(\R^d)$, i.e., $P\left( \fB_t \in G^2(\R^d), \, t
  \in [0,1] \right) = 1$. In a similar way, we will consider the (step-$m$
truncated) \emph{signature}, see Friz and Victoir~\cite{fri/vic10},
\begin{equation}
  \label{eq:signature}
  S_m(B)_{0,t} \coloneqq 1 + \sum_{k=1}^m \sum_{i_1, \ldots, i_k \in \{1,
    \ldots, d\}} \int_{0 \le t_1 \le \cdots \le t_k \le t} \circ
  dB^{i_1}_{t_1} \cdots \circ dB^{i_k}_{t_k}\, e_{i_1} \otimes \cdots \otimes
  e_{i_k},
\end{equation}
which takes values in the step-$m$ nilpotent Lie-group $G^m(\R^d)$ defined
analogously to $G^2(\R^d)$.\footnote{Note that $S_m(B)_{0,\cdot} =
  S_m(\fB)_{0,\cdot}$, the \emph{Lyons lift} of the enhanced Brownian motion,
  reflecting the fact that $S_m(B)$ depends uniquely and continuously on $\fB$
  -- whereas $\fB$ itself is not uniquely and certainly not continuously given
  by $B$. For instance, we could have chosen the Ito-integral instead of the
  Stratonovich integral.}

Consider the stochastic differential equation (in Stratonovich form)
\begin{equation}
  \label{eq:sde}
  dX_t = V_0( X_t) dt + \sum_{i=1}^d V_i(X_t) \circ dB^i_t,
\end{equation}
$X_0 = x_0 \in \R^N$. Here, $V_0, V_1, \ldots, V_d: \R^N \to \R^N$ is a
collection of smooth vector fields. A \emph{cubature formula on Wiener space}
is a random variable $W$ taking values in the space $\Cbv([0,1], \R^d)$ of
continuous paths of bounded variation with values in $\R^d$ such that we have
\begin{equation}
  \label{eq:cubature-formula}
  E \left[ \int_{0 \le t_1 \le \cdots \le t_k \le 1} \circ dB^{i_1}_{t_1}
    \cdots \circ dB^{i_k}_{t_k} \right] = E\left[ \int_{0 \le t_1 \le
      \cdots \le t_k \le 1} dW^{i_1}_{t_1} \cdots dW^{i_k}_{t_k} \right].
\end{equation}
Equation~(\ref{eq:cubature-formula}) is supposed to hold for all multi-indices
$I = (i_1, \ldots, i_k) \in \{1,\ldots,d\}^k$ with $k \le
m$ and all $1 \le k \le m$, where $m$ is a fixed positive integer, the
\emph{order} of the cubature formula. Moreover, we note that
the paths of the process $W$ are of bounded variation, therefore the integrals
on the right hand side of~(\ref{eq:cubature-formula}) can be classically defined
in a pathwise sense. Notice that we do not use cross-integrals between time
$dt$ and the Brownian motion $dB_t$. Therefore, a cubature formula in this
sense can only be used to approximate SDEs with drift $V_0 \equiv 0$. We will
cover the general case later in Section~\ref{sec:main-result}.

Rephrased in terms of the (truncated) signature,
equation~(\ref{eq:cubature-formula}) means that
\begin{equation}
  \label{eq:cubature-formula-signature}
  E\left[S_m(B)_{0,1}\right] = E\left[S_m(W)_{0,1}\right],
\end{equation}
where the expectation takes values in the algebra $T^m(\R^d)$. Obviously, any
cubature formula on Wiener space can be rescaled to a cubature formula on the
interval $[0,\dt]$, $\dt > 0$, by replacing $W$ with the bounded variation
path
\begin{equation}
  \label{eq:rescaled-cubature-formula}
  \delta_{\sqrt{\dt}}(W): [0,\dt] \to \R^d, \quad s \mapsto \sqrt{\dt}
  W(s/\dt).
\end{equation}
On the level of signatures, this corresponds to applying the \emph{dilatation}
operator $\delta_{\sqrt{\dt}}: G^m(\R^d) \to G^m(\R^d)$, i.e.,
\begin{equation*}
  S_m\left( \delta_{\sqrt{\dt}}(W) \right)_{0,\dt} = \delta_{\sqrt{\dt}}(
  S_m(W)_{0,1} ).
\end{equation*}
\begin{remark}
  Note that the symbol $\delta_{\sqrt{\dt}}$ has different meanings on both
  sides of the equation: on the left hand side, it is a function from
  $C([0,1], \R^d)$ to $C([0,\dt], \R^d)$, whereas on the right hand side it is
  the restriction to $G^m(\R^d)$ of a linear map defined on the algebra
  $T^m(\R^d)$ by
  \begin{equation*}
    \delta_{\sqrt{\dt}}\left( e_{i_1} \otimes \cdots \otimes e_{i_k} \right)
    \coloneqq \dt^{k/2} e_{i_1} \otimes \cdots \otimes e_{i_k}, \quad 0 \le k
    \le m, \quad i_1, \ldots, i_k \in \{1, \ldots, d\}.
  \end{equation*}
\end{remark}

Given a mesh $\mathcal{D} = \{0 = t_0 < t_1 < \cdots < t_n = 1\}$, set $\dt_k
\coloneqq t_k - t_{k-1}$, $k = 1, \ldots, n$, and $\abs{\mathcal{D}} = \max_k
\dt_k$. Moreover, let $W_{(1)}, \ldots, W_{(n)}$ be independent copies of the
cubature formula $W$. We define a random variable $W^{\mathcal{D}}: [0,1] \to
\R^d$ taking values in the space of continuous paths of bounded variation by
concatenation of the paths $\delta_{\sqrt{\dt_k}}(W_{(k)}): [0, \dt_k] \to
\R^d$, $k = 1, \ldots, n$. Again, by well known properties of the signature
(the Chen theorem, see for instance~\cite[Theorem 7.11, Exercise
7.14]{bre/fri/hue09}), this translates to the relation
\begin{align*}
  S_m\left( W^{\mathcal{D}} \right)_{0,1} &= S_m\left(
    \delta_{\sqrt{\dt_1}}(W_{(1)}) 
  \right)_{0,\dt_1} \otimes \cdots \otimes S_m\left( \delta_{\sqrt{\dt_n}}(W_{(n)})
  \right)_{0,\dt_n} \\
  &= \delta_{\sqrt{\dt_1}} \left(S_m( W_{(1)} )_{0,1}\right) \otimes
  \cdots \otimes \delta_{\sqrt{\dt_n}} \left( S_m( W_{(n)} )_{0,1} \right),
\end{align*}
where $\otimes$ denotes the multiplication in the Lie group $G^m(\R^d)$.
Finally, let $X^{\mathcal{D}}$ denote the (pathwise ODE) solution of the
equation
\begin{equation}
  \label{eq:cubature-ODE}
  dX^{\mathcal{D}}_t = \sum_{i=1}^d V_i(X^{\mathcal{D}}_t) dW^{\mathcal{D},i}_t,
\end{equation}
$X^{\mathcal{D}}_0 = x_0$. For a given function $f: \R^N \to \R$ of interest,
the method of cubature on Wiener space now consists in the approximation
\begin{equation}\label{eq:cubature-approximation}
  E[f(X_1)] = E\left[f(X^{\mathcal{D}}_1)\right] + \mathcal{O}\left(
    \abs{D}^{(m-1)/2} \right),
\end{equation}
provided that certain
regularity assumptions are satisfied, see~\cite{lyo/vic04},~\cite{nin/vic08}
and~\cite{kus04}. In particular, the method provides an efficient numerical
scheme, if $W$ has been chosen in such a way that integration
of~(\ref{eq:cubature-ODE}) is ``substantially simpler'' then integration of
the original~(\ref{eq:sde}), see~\cite{bay/fri/loe12}. If $f$ is
smooth,~\eqref{eq:cubature-approximation} holds even for uniform meshes. If
$f$ only is Lipschitz, however, then Kusuoka~\cite[Theorem 4]{kus04} shows
that~\eqref{eq:cubature-approximation} holds provided that one takes certain
non-homogeneous meshes. The goal of this paper regarding cubature is to show
that convergence even holds for (reasonable) functionals $f$ depending on the
whole path $(X_t)_{0 \le t \le 1}$.

\begin{example}
  \label{ex:cubature-lyons-victoir}
  The cubature formulas in~\cite{lyo/vic04} are discrete random variables $W$
  taking values in the space of continuous paths of bounded variations. That
  is, fix $k$ paths of bounded variation $\omega_1, \ldots, \omega_k: [0,1]
  \to \R^d$ and positive real numbers $\lambda_1, \ldots, \lambda_k$ with
  $\lambda_1 + \cdots + \lambda_k = 1$. Then $W$ is the random variable taking
  values in $\{\omega_1, \ldots, \omega_k\}$ with $P(W = \omega_j) =
  \lambda_j$, $j = 1, \ldots, k$. In all the concrete cubature formulas
  constructed in~\cite{lyo/vic04}, the paths $\omega_j(\cdot)$ are, in fact,
  piecewise linear. 
\end{example}

\begin{example}
  \label{ex:wong-zakai-cubature}
  We can even interpret the Wong-Zakai approximation as a cubature formula on
  Wiener space (of order $m=3$; the resulting convergence in
  ~\eqref{eq:cubature-approximation} has then weak order $(m-1)/2=1$,
  precisely as the usual Euler scheme for Ito differential equations). Indeed,
  choose $W$ as the linear path $W_t = t 
  B_1$. We note that $W^{\mathcal{D}}$ can be realized (for any mesh
  $\mathcal{D}$) by choosing $\delta_{\sqrt{\dt_k}}(W_{(k)})(s) =
  s(B_{t_k}-B_{t_{k-1}})$, because $W_{(k)}(t) \coloneqq \f{t}{\sqrt{\dt_k}}
  (B_{t_k}-B_{t_{k-1}})$ has the same law as $W$ and all the $W_{(k)}$ are
  independent. Concatenation of these paths precisely gives the piecewise
  linear approximation of $B$ with nodes in $\mathcal{D}$.
\end{example}

\begin{example}
  \label{ex:cubature-ninomiya-victoir}
  Ninomiya and Victoir~\cite{nin/vic08} construct a cubature formula of order
  $m = 5$ in the following way. Let $\Lambda$ be a Bernoulli random variable
  (taking values $\pm 1$ with probability $1/2$ each) and let $Z^1, \ldots,
  Z^d$ be independent standard normal random variables. Set $\varepsilon =
  1/(d+1)$. For $\omega \in \Omega$, $W(\omega)$ is defined by the following
  formula. If $\Lambda(\omega) = -1$, we define $W(\omega)$ to be the
  piecewise linear path with
  \begin{equation*}
    \dot{W}^i(\omega)(s) =
    \begin{cases}
      1/\varepsilon, & s \in [0,\varepsilon/2], \ i = 0,\\
      Z^i(\omega)/\varepsilon, & s \in ]\varepsilon/2 + (i-1)\varepsilon,
      \varepsilon/2+i\varepsilon], \ i \in \{1, \ldots, d \}, \\
      1/\varepsilon, & s \in ]1-\varepsilon/2,1], \ i = 0,\\
      0, & \text{else}.
    \end{cases}
  \end{equation*}
  If $\Lambda(\omega) = 1$, $W(\omega)$ is similarly defined by
  \begin{equation*}
    \dot{W}^i(\omega)(s) =
    \begin{cases}
      1/\varepsilon, & s \in [0,\varepsilon/2], \ i = 0,\\
      Z^i(\omega)/\varepsilon, & s \in ]\varepsilon/2 + (d-i)\varepsilon,
        \varepsilon/2+(d-i+1)\varepsilon], \ i \in \{1, \ldots, d \}, \\
      1/\varepsilon, & s \in ]1-\varepsilon/2,1], \ i = 0,\\
      0, & \text{else}.
    \end{cases}
  \end{equation*}
  This means, we subdivide the interval $[0,1]$ into $d+2$ subintervals
  \begin{equation*}
    \left[0,\f{\vphantom{3}\varepsilon}{2}\right] \cup \left]\f{\varepsilon}{2},
        \f{3\varepsilon}{2}\right] \cup \cdots \cup \left] 1-
        \f{3\varepsilon}{2}, 1-\f{\varepsilon}{2}\right] \cup
      \left]1-\f{\vphantom{3}\varepsilon}{2}, 1\right]. 
  \end{equation*}
  On each of these subintervals, $W(\omega)$ is constant in all components
  albeit one, which is linear. In particular, $W(\omega)$ is again piecewise
  linear.

  At this stage, we would like to remark that we could replace the Gaussian
  random variables $Z^i$ by discrete random variables having the same moments
  of order up to five. Then we would obtain a special case of
  Example~\ref{ex:cubature-lyons-victoir} -- albeit for the non-standard
  choice of $W^0$, see Section~\ref{sec:main-result} below.
\end{example}

Let us now turn our attention to Donsker type results: for a fixed sequence of
meshes $\mathcal{D}_n$ with $\abs{\mathcal{D}_n} \to 0$ we wish to study the
corresponding sequence of paths in $G^m(\R^d)$, i.e., we study
\begin{equation*}
  S_m\left( W^{\mathcal{D}_n} \right)_{0,\cdot} = \left( S_m\left(
      W^{\mathcal{D}_n} \right)_{0,t} \right)_{t \in [0,1]}.
\end{equation*}
By a \emph{Donsker theorem in rough path topology} for the sequence
of cubature formulas $W^{\mathcal{D}_n}$ we understand the statement that
\begin{equation}
  \label{eq:donsker-cubature}
  S_m\left( W^{\mathcal{D}_n} \right)_{0,\cdot} \xrightarrow[n \to \infty]{}
  S_m( \fB )_{0,\cdot}
\end{equation}
weakly with respect to $\alpha$-H\"{o}lder rough path topology\footnote{I.e.,
  the $\alpha$-H\"older topology for functions taking values in the metric
  space $(G^2(\R^d), \norm{\cdot})$.,
see~\cite[Definition 5.1, 9.15]{fri/vic10}}, for some $\alpha \in (1/3, 1/2)$
and $m \ge 2$. (In fact, elementary results of rough path theory imply then
that it suffices to consider $m=2$. Also, the claimed convergence will
actually be established for {\it all} $\alpha < 1/2$). As a justification for
calling the convergence stated in~(\ref{eq:donsker-cubature}) a Donsker
theorem, consider the following random walk. Let us again fix the mesh
$\mathcal{D}_n = \{0= t_0 < \cdots < t_n = 1 \}$. (We only take $n$ as the
number of sub-intervals for the grid $\mathcal{D}_n$ for more convenient
notation. The mathematics would, of course, work in precisely the same way, if
the size of $\mathcal{D}_n$ was completely arbitrary, as long as
$\abs{\mathcal{D}_n} \to 0$.) Define
\begin{equation}
  \label{eq:def-xik}
  \xi^n_k = S_m\left( \delta_{\sqrt{\dt_k}}( W_{(k)} ) \right)_{0,\dt_k} =
  \delta_{\sqrt{\dt_k}}(S_m(W_{(k)})_{0,1}),
\end{equation}
a random variable taking values in $G^m(\R^d)$. Note that $\xi^n_k =
\delta_{\sqrt{\dt_k}}(\xi_{(k)})$, where $\xi_{(k)}$ is an independent copy of
$S_m(W)_{0,1}$. Next define the $G^m(\R^d)$-valued, finite random walk
$\Xi^n_k$, $k = 0, \ldots,n$, by $\Xi^n_0 = 1$ and
\begin{equation*}
  \Xi^n_{k+1} = \Xi^n_k \otimes \xi^n_{k+1},
\end{equation*}
where $1$ is the neutral element of $G^m(\R^d)$. Since
\begin{equation*}
  S_m \left( W^{\mathcal{D}_n} \right)_{0,t_k} = \Xi^n_k, \quad k = 0, \ldots, n,
\end{equation*}
$S_m\left( W^{\mathcal{D}_n} \right)_{0,\cdot}$ is, indeed, a path in
$G^m(\R^d)$ obtained from the random walk $\Xi^n$ by (possibly random)
interpolation. This gives the link to the classical Donsker theorem as well as
to the paper of Breuillard, Friz and Huesmann~\cite{bre/fri/hue09}. Let us
rephrase their Theorem~3 for the current setting.
\begin{proposition}
  \label{prop:1}
  Let $W$ be a cubature formula on Wiener space of order $m = 2$ with finite
  moments of all orders in the sense that
  \begin{equation*}
    \forall q \ge 1: E\left[\norm{S_2(W)_{0,1}}^q\right] < \infty,
  \end{equation*}
  where $\norm{\cdot}$ denotes the Carnot-Caratheodory norm on $G^2(\R^d)$,
  see below. Moreover, assume that $W$ is chosen in such a way
  that for every $\omega$, $S_2(W(\omega))_{0,\cdot}$ is a geodesic connecting
  $1$ and $S_2(W(\omega))_{0,1}$. Choose uniform meshes $\mathcal{D}_n =
  \left\{\left.\f{k}{n} \,\right|\, k = 0, \ldots, n \right\}$. Then the
  Donsker theorem holds in rough path topology, i.e.,
  $S_2\left(W^{\mathcal{D}_n}\right)_{0,\cdot}$ converges to $\fB$ in
  $C^{0,\alpha-\text{H\"{o}l}}( [0,1], G^2(\R^d))$, for every $\alpha < 1/2$.
\end{proposition}
Recall that the Carnot-Caratheodory norm is defined by
\begin{equation*}
  \norm{\mathbf{x}} \coloneqq \inf \Set{\int_0^1 \abs{d\gamma} | \gamma \in
    C^{1-var}\left( [0,1], \R^d \right),\, S_2(\gamma)_{0,1} = \mathbf{x}}.
\end{equation*}
The infimum is actually always attained, and can be parametrized as
Lipschitz-continuous path with constant speed, i.e., $\norm{S_2(\gamma)_{0,t}}
= t \norm{S_2(\gamma)_{0,1}}$ for $0<t<1$ and a minimizing path $\gamma$,
see~\cite[Theorem 7.33]{fri/vic10}. As a homogeneous norm, the
Carnot-Caratheodory norm is equivalent to the simpler norm
\begin{equation*}
  \norm{\mathbf{x}}_2 \coloneqq \max\left( \abs{x}, \abs{a}^{1/2} \right),
  \quad \mathbf{x} = 1 + x + a \in G^2(\R^d),
\end{equation*}
see~\cite[Theorem 7.45]{fri/vic10}.

\begin{remark}
  \label{rem:2}
  In~\cite[Theorem~1]{bre/fri/hue09}, the moment condition is relaxed,
  which gives weak convergence in $\alpha$-H\"{o}lder norm for all $\alpha <
  \alpha^\ast$, for some $\alpha^\ast < 1/2$, which is related to the relaxed
  moment condition. In this paper, we shall always assume existence of all
  the moments. We note, however, that we could also relax this assumption,
  obtaining a similar result.
\end{remark}

\section{The main result}
\label{sec:main-result}

Usually, a cubature formula $W$ will not satisfy the conditions of
Proposition~\ref{prop:1}, even if we only choose uniform meshes, because the
corresponding interpolation $S\left( W^{\mathcal{D}_n} \right)_{0,\cdot}$ of
the random walk $\Xi^n$ will not be geodesic. Moreover, if we want to treat
functions $f$ which are not smooth, then we have to choose non-uniform meshes
with $t_k = \f{k^\gamma}{n^\gamma}$ for some $\gamma > m-1$,
see~\cite{kus04}. Therefore, we want to generalize Proposition~\ref{prop:1} in
two directions. We want to get rid of the condition of geodesic interpolation,
and we want to generalize to non-uniform meshes. Fortunately, the first
generalization is simple, at least for the cubature formulas actually
suggested in the literature, see
Example~\ref{ex:cubature-lyons-victoir},~\ref{ex:cubature-ninomiya-victoir}
and also for the Wong-Zakai approximation given in
Example~\ref{ex:wong-zakai-cubature}. The second generalization, however,
requires us to change the method of proof as compared to~\cite{bre/fri/hue09}.

It is natural to impose some restriction on the behavior of $S\left(
  W^{\mathcal{D}_n} \right)_{0,\cdot}$ between two nodes of the random
walk. Indeed, we have to rule out ``loops'' which approach infinity. 

\begin{assumption}
  \label{ass:interpolation}
  The cubature formula $W$ takes values in the Cameron-Martin space
  $\mathcal{H}$ (of paths started at $0$) and the Cameron-Martin norm has
  finite moments of all orders, i.e., for every $k \in \N$
  \begin{equation*}
    E\left[ \norm{W}_{\mathcal{H}}^k \right] = E\left[ \left( \int_0^1
        \abs{\dot{W}(s)}^2 ds \right)^{k/2} \right] < \infty.
  \end{equation*}
\end{assumption}

This assumption is both natural (a general continuous path of finite
$1$-variation would not be in the (1/2-$\epsilon$)-H\"older support of the
Wiener measure!) and satisfied by all (piecewise linear!) cubature formulas
used in practice. We look at this in some detail in
\begin{example}
  \label{ex:piecewise-linear-cubature}
  Assume that the cubature formula $W$ is piecewise linear, i.e., there is a
  positive integer $\ell$ and there are $d$-dimensional random variables $F_1,
  \ldots, F_\ell$ with finite moments of all orders and a mesh $0 = s_0 <
  \cdots < s_\ell = 1$ such that
   \begin{equation*}
     \dot{W}_s = F_l, \quad s_{l-1} < s \le s_l,\ 1 \le l \le \ell.
   \end{equation*}
   This immediately implies Assumption~\ref{ass:interpolation}.
\end{example}


Our main theorem is (for conclusions to cubature see
Corollary~\ref{cor:weak-convergence} below):
\begin{theorem}
  \label{thr:1}
  Given a cubature formula $W$ of order $m \ge 2$ such that $W_1$ and the
  corresponding area $A_1$ have finite moments of all orders and
  Assumption~\ref{ass:interpolation} is satisfied. Then Donsker's theorem
  holds in rough path topology for any sequence $\mathcal{D}_n$ of meshes with
  $\abs{\mathcal{D}_n} \to 0$, i.e.,
  \begin{equation*}
    S_2\left( W^{\mathcal{D}_n} \right)_{0,\cdot} \xrightarrow[n \to \infty]{}
    S_2( \fB )_{0,\cdot}
  \end{equation*}
  in $C^{0,\alpha-\text{H\"{o}l}}( [0,1], G^{2}(\R^d))$, for every $\alpha <
  1/2$.
\end{theorem}

The natural conclusion from Theorem~\ref{thr:1} would be a weak convergence
result for the cubature-approximation of the SDE~\eqref{eq:sde} to its true
solution on path-space. A little care is necessary, however, 
because we have ignored the drift $V_0$ in the SDE, i.e., our driving signal
is a pure Brownian motion and does not include time. The classical approach is
to add another component to both the Brownian motion and the approximating
cubature paths by setting $B^0_t \coloneqq t$, $W^0_t \coloneqq t$ and then
require the moment matching condition~\eqref{eq:cubature-formula} to hold for
all iterated integrals, where the multi-index $(i_1,\ldots,i_k)$ now varies
over $\{0,1,\ldots,d\}^k$, i.e., where we also consider mixed iterated
integrals of Brownian motion and time $t$. Due to the scaling of Brownian
motion ``$dB_t \approx \sqrt{dt}$'', it is only necessary to impose the moment
matching condition for multi-indices $(i_1,\ldots,i_k)$ with $k +
\#\{j|i_j=0\} \le m$ to get weak convergence with rate $\f{m-1}{2}$.

However, the Ninomiya-Victoir scheme does not fall into this class, because we
have seen in Example~\ref{ex:cubature-ninomiya-victoir} that they do not
choose $W^0_t \equiv t$. Therefore, we want to generalize the above
considerations slightly. Let $h:[0,1] \to \R$ be a deterministic, uniformly
Lipschitz path with $h(0) = 0$ and $h(1) = 1$. This setting obviously includes
the drift-component of the Ninomiya-Victoir scheme. We define the path
$\Prefix^{h}{W}$ by $\Prefix^{h}{W}^i_t \coloneqq W^i_t$ for $i=1,\ldots,d$,
and $\Prefix^{h}{W}^0_t \coloneqq h(t)$. As usual, we set $B^0_t \coloneqq
t$. We assume the usual moment matching condition to hold, i.e.,
\begin{equation*}
  E\left[ S_{m}(\Prefix^{h}{W})_{0,1} \right] = E\left[ S_{m}(B)_{0,1}\right],
\end{equation*}
where $S_m(\Prefix^{h}{W})$ is the step-$m$ signature of the path
$\Prefix^{h}{W}$, more precisely
\begin{equation*}
  S_m\left( \Prefix^{h}{W} \right)_{0,1} = \sum_{k=0}^m \sum_{\substack{(i_1,
      \ldots, i_k) \in \{ 0,1, \ldots, d\}^k\\k+\#\{j\,:\,i_j = 0\} \le m}}
  \int_{0 \le t_1 \le \cdots \le t_k \le 1} d\Prefix^{h}{W}^{i_1}_{t_1} \cdots
  d\Prefix^{h}{W}^{i_k}_{t_k}\, e_{i_1} \otimes \cdots \otimes e_{i_k}. 
\end{equation*}
Analogously, the signature of the Brownian motion above is understood as the
signature of the now $\R^{d+1}$-valued process $B$. (This notation is
ambiguous. In the following, the symbol $B$ will usually denote the
$\R^d$-valued Brownian motion. We only mean the extended $\R^{d+1}$-valued
process if specifically indicated.) We note that the signatures of the
$(d+1)$-dimensional processes take their values in a stratified Lie group
denoted by $G^m_1(\R^d)$. In a similar fashion as above, we obtain -- by
rescaling and concatenation -- a stochastic process
$\Prefix^{h}{W}^{\mathcal{D}}$ along a grid $\mathcal{D}$. Of course, we have
to use a different rescaling for the component $\Prefix^{h}{W}^0$. Indeed,
following the construction in Section~\ref{sec:cubat-wien-space}, we define
$\delta_{\sqrt{\dt}}(\Prefix^{h}{W}):[0,\dt] \to \R^{d+1}$ by
$\delta_{\sqrt{\dt}}(\Prefix^{h}{W})^i_s = \sqrt{\dt}W^i_{s/\dt}$ for $i = 1,
\ldots, d$, as before, but $\delta_{\sqrt{\dt}}(\Prefix^{h}{W})^0_s = \dt\,
h(s/\dt)$. We continue to construct $\Prefix^{h}{W}^{\mathcal{D}}$ by
concatenation.

By rough path theory (continuity of Young pairing, e.g.~\cite[Section
9.4.4]{fri/vic10}) we also obtain weak convergence in path-space for the
extended process $\Prefix^{h}{W}$ to the extended, $\R^{d+1}$-valued Brownian
motion, which even holds for any truncated signature -- not only for the
step-$2$ signature.
\begin{corollary}
  \label{cor:donsker-with-drift}
  Let $W$ be cubature formula on Wiener space of order $m \ge 2$ satisfying
  Assumption~\ref{ass:interpolation} and such that $W_1$ and the corresponding
  area $A_1$ have finite moments of all orders. Then Donsker's theorem
  holds in rough path topology for any sequence $\mathcal{D}_n$ of meshes with
  $\abs{\mathcal{D}_n} \to 0$, i.e.,
  \begin{equation*}
    S_N\left( \Prefix^{h}{W}^{\mathcal{D}_n} \right)_{0,\cdot} \xrightarrow[n
    \to \infty]{} S_N( B )_{0,\cdot}
  \end{equation*}
  in $C^{0,\alpha-\text{H\"{o}l}}( [0,1], G^{N}_1(\R^d))$, for every $\alpha <
  1/2$ and any $N \ge 1$.
\end{corollary}

Moreover, we define $\Prefix^{h}{X}^{\mathcal{D}}$ as the solution to the
(random) ODE
\begin{equation}
  \label{eq:6}
  d\Prefix^{h}{X}^{\mathcal{D}}_t = V_0\left( \Prefix^{h}{X}^{\mathcal{D}}_t
  \right) dh^{\mathcal{D}_n}(t) + \sum_{i=1}^d V_i\left(
    \Prefix^{h}{X}^{\mathcal{D}}_t \right) dW^{\mathcal{D},i}_t,
\end{equation}
where $h^{\mathcal{D}_n} \coloneqq \Prefix^{h}{W}^{\mathcal{D}_n,0}$.
Then we have weak convergence of $\Prefix^{h}{X}^{\mathcal{D}_n}$ to $X$ on
path-space.

\begin{corollary}
  \label{cor:weak-convergence}
  Given a bounded, continuous functional $f: C^{0,\alpha-\text{H\"{o}l}}(
  [0,1], \R^N) \to \R$, and assume that $W$, $h$ and $\mathcal{D}_n$ satisfy
  the assumptions of Corollary~\ref{cor:donsker-with-drift}. Then we have
  \begin{equation*}
    E\left[ f\left( \Prefix^{h}{X}^{\mathcal{D}_n} \right) \right]
    \xrightarrow[n \to \infty]{} E[f(X)],
  \end{equation*}
  where $X$ denotes the path $(X_t)_{t \in [0,1]}$ of the true solution of the
  SDE~(\ref{eq:sde}) and $\Prefix^{h}{X}^{\mathcal{D}_n}$ denotes the pathwise
  solution of the ODE~\eqref{eq:6}.
\end{corollary}
\begin{proof}
  We interpret~(\ref{eq:cubature-ODE}) as a rough differential equation, i.e.,
  for a given (rough) path $\fomega \in
  C^{0,\alpha-\text{H\"ol}}([0,1];G^2_1(\R^d))$ with $\omega = \pi_1(\fomega)$
  we define $\pi(\fomega)_t \coloneqq y_t$ by
  \begin{equation*}
    dy_t = V_0(y_t) d\omega^0_t + \sum_{i=1}^d V_i(y_t) d\omega^i_t.
  \end{equation*}
  In particular, we have $\Prefix^{h}{X}^{\mathcal{D}_n} = \pi\left(
    S_2(\Prefix^{h}{W}^{\mathcal{D}_n})_{0,\cdot}\right)$ and $X =
  \pi(S_2(B)_{0,\cdot})$. By~\cite[Theorem 10.26]{fri/vic10}, the map $\fomega
  \mapsto \pi(\fomega)_\cdot$ is a continuous map from
  $C^{0,\alpha-\text{H\"ol}}([0,1];G^2_1(\R^d))$ to
  $C^{0,\alpha-\text{H\"ol}}([0,1];\R^N)$. Thus,
  Corollary~\ref{cor:donsker-with-drift} implies weak convergence of
  $\Prefix^{h}{X}^{\mathcal{D}_n} = \pi\left(
    S_2(\Prefix^{h}{W}^{\mathcal{D}_n})_{0,\cdot} \right)$ to $X =
  \pi(S_2(B)_{0,\cdot})$ in $C^{0,\alpha-\text{H\"ol}}([0,1];\R^N)$.
\end{proof}
\begin{remark}
  Since the $\alpha$-H\"older topology is stronger than the usual uniform
  topology given by the supremum norm, Corollary~\ref{cor:weak-convergence} in
  particular holds for all bounded functionals $f$ which are continuous in the
  uniform topology on path space. In the case of unbounded continuous
  functionals, convergence can still be guaranteed provided that some uniform
  integrability property holds. (Of course, in the case of call-option type
  derivatives, one could also try a relevant put-call-parity.) Finally, in the
  case of barrier options, the payoff functional is often continuous apart
  from a set of measure zero on path space. Naturally, non-continuities on
  null-sets do not hinder weak convergence of the cubature method.
\end{remark}

\section{Random walks with independent, non-identically distributed increments}
\label{sec:random-walks-with}

In this section, we prepare the main ingredients of a proof of Donsker's
theorem for random walks with independent, but not identically distributed
increments on the Lie group $G \coloneqq G^2(\R^d)$. More precisely, let $\xi$
be a random variable with values in $G$ with finite moments of all orders. We
shall denote the components of $\xi$ in the basis of $\G \coloneqq \G^2(\R^d)$
given by $e_i$, $1\le i \le d$, together with $[e_i,e_j]$, $1 \le i < j \le
d$, by $X^i$ and $A^{i,j}$, respectively, i.e.,
\begin{equation*}
  \xi = \exp\left( \sum_{i=1}^d X^i e_i + \sum_{i<j} A^{i,j} [e_i,e_j] \right).
\end{equation*}
Thus, the condition that $\xi$ has finite moments of all orders simply means
that all the real random variables $X^i$, $A^{i,j}$ have finite moments of all
orders $q \ge 1$. Moreover, we assume that $\xi$ is \emph{centered}, i.e.,
\begin{equation*}
  E[X^i] = 0, \quad i= 1, \ldots, d.
\end{equation*}

Let us fix a mesh $\mathcal{D}_n = \{0 = t_0 < \cdots < t_n = 1\}$. For $n$
independent copies $\xi_{(1)}, \ldots, \xi_{(n)}$ of $\xi$, define $\xi^n_k =
\delta_{\sqrt{\dt_k}}( \xi_{(k)} )$ and the corresponding random walk
\begin{equation*}
  \Xi^n_0 = 1, \quad \Xi^n_{k} = \Xi^n_{k-1} \otimes \xi^n_k, \ k = 1, \ldots, n.
\end{equation*}
For use in the next lemma, let us define the coordinate mappings $x^i$
(mapping $x \in G$ to the component of $\log(x)$ with respect to the basis
element $e_i$) and $x^{i,j}$ (mapping $x \in G$ to the component of $\log(x)$
with respect to the basis element $[e_i,e_j]$), $1 \le i \le d$, $i<j\le
d$. As usual, the corresponding vector-fields (i.e., basis of the tangent
space) are denoted by $\f{\pa}{\pa x^i}$ and $\f{\pa}{\pa x^{i,j}}$,
respectively.
\begin{lemma}
  \label{lem:clt-non-iid}
  The above random walk satisfies the central limit theorem, i.e., $\Xi^n_n$
  converges weakly to the Gaussian measure with infinitesimal generator
  \begin{equation*}
    \sum_{i<j} a^{i,j} \f{\pa}{\pa x^{i,j}} + \f{1}{2} \sum_{i \le j} b^{i,j}
    \f{\pa}{\pa x^i} \f{\pa}{\pa x^j},
  \end{equation*}
  where $a^{i,j} \coloneqq E[A^{i,j}]$ and $b^{i,j} \coloneqq
  \operatorname{Cov}(X^i, X^j)$.\footnote{The statement means that there is a
    semi-group $(\mu_t)_{t\ge0}$ of probability measures on $G$ having the
    above infinitesimal generator and $\mu_1$ is the limiting distribution of
    $\Xi^n_n$. Moreover, this semi-group is Gaussian in the sense that
    $\lim_{t\searrow 0} \f{1}{t} \mu_t(G\setminus U) = 0$ for every
    neighborhood $U$ of the neutral element of the group $G$.}
\end{lemma}
\begin{proof}
  The result is well-known in probability theory on Lie groups, see,
  e.g.,~\cite{pap93}. We verify that the system of probability measures
  $\mu_{n,k} = (\xi^n_k)_\ast P$, i.e., $\mu_{n,k}$ is the law of $\xi^n_k$,
  satisfies the conditions given in~\cite[Theorem 3.2]{pap93}, namely:
  \begin{enumerate}
  \item[(i)] $\sup_n \sum_{k=1}^n \int_G \norm{x}^2 \mu_{n,k}(dx) < \infty$;
  \item[(ii)] $\mu_{n,k}$ is centered in the above sense;
  \item[(iii)] for every $1 \le i < j \le d$, there is a number $a^{i,j} \in
    \R$ such that 
    \begin{equation*}
      a^{i,j} = \lim_{n\to\infty} \sum_{k=1}^n \int_G x^{i,j}(x)
      \mu_{n,k}(dx);
    \end{equation*}
  \item[(iv)] for every $1 \le i,j \le d$, there is a number $b^{i,j} \in \R$
    such that
    \begin{equation*}
      b^{i,j} = \lim_{n\to\infty} \sum_{k=1}^n \int_G x^{i}(x) x^j(x)
      \mu_{n,k}(dx);
    \end{equation*}
  \item[(v)] $\lim_{n\to\infty} \sum_{k=1}^n \int_{\norm{x}\ge \epsilon}
    \norm{x}^2 \mu_{n,k}(dx) = 0$ for all $\epsilon > 0$.
  \end{enumerate}
  
  By homogeneity of the Carnot-Caratheodory norm, we have
  \begin{equation*}
    \int_G \norm{x}^2 \mu_{n,k}(dx) = E\left[
      \norm{\delta_{\sqrt{\dt_k}}(\xi)}^2 \right] = \dt_k E\left[ \norm{\xi}^2
    \right]. 
  \end{equation*}
  Thus,
  \begin{equation*}
    \sum_{k=1}^n \int_G \norm{x}^2 \mu_{n,k}(dx) = \sum_{k=1}^n \dt_k E\left[
      \norm{\xi}^2 \right] =  E\left[ \norm{\xi}^2 \right],
  \end{equation*}
  and the supremum over $n$ is obviously finite, settling (i).

  (ii) is satisfied by assumption on $\xi$. Regarding (iii), note that
  $x^{i,j}(\xi^n_k) = x^{i,j}(\delta_{\sqrt{\dt_k}}(\xi)) = \dt_k A^{i,j}$,
  where equality is understood as equality in law. Therefore, (iii) is
  satisfied with $a^{i,j} = E(A^{i,j}) < \infty$. A similar argument shows
  that (iv) holds with $b^{i,j} = \operatorname{Cov}(X^i, X^j)$.

  For the proof of (v), we again use homogeneity of the Carnot-Caratheodory
  norm. Indeed, we have
  \begin{equation*}
    \int_{\norm{x} \ge \epsilon} \norm{x}^2 \mu_{n,k}(dx) = \dt_k E\left[
      \mathbf{1}_{]\epsilon, \infty[}(\sqrt{\dt_k} \norm{\xi}) \norm{\xi}^2
    \right], 
  \end{equation*}
  implying that
  \begin{equation*}
    \sum_{k=1}^n \int_{\norm{x}\ge \epsilon} \norm{x}^2 \mu_{n,k}(dx) \le
    E\left[ \mathbf{1}_{\left] \f{\epsilon}{\sqrt{\abs{\mathcal{D}_n}}},
          \infty \right[} \norm{\xi}^2 \right] \le \sqrt{P\left( \norm{\xi} >
        \f{\epsilon}{\sqrt{\abs{\mathcal{D}_n}}}\right) } \cdot \sqrt{E\left[
        \norm{\xi}^4 \right]},
  \end{equation*}
  by the Cauchy-Schwarz inequality. Now, the right hand side converges to zero
  by integrability of $\norm{\xi}$ and $\abs{\mathcal{D}_n} \to 0$, for every
  fixed $\epsilon > 0$.
\end{proof}
\begin{remark}
  If $\xi$ is the step-$2$ signature of a cubature formula of degree $m \ge
  2$, then $a^{i,j} = E[A^{i,j}] = 0$, $1 \le i < j \le d$, and, moreover,
  $b^{i,j} = \operatorname{Cov}(X^i,X^j) = \delta_{ij}$. Thus, the
  generator of the limiting Gaussian measure in Lemma~\ref{lem:clt-non-iid}
  coincides with the generator of the Brownian motion on $G$, i.e., with the
  generator of $\fB$.
\end{remark}

Next we state a moment estimate, which will enable us to prove tightness of
the family of interpolated random walks in rough path topology.
\begin{proposition}
  \label{prop:2}
  For every $p \in \N$, $p \ge 1$ we can find a constant $C$ independent of
  $k$ and $n$ such that
  \begin{equation*}
    E\left[ \norm{\Xi^n_k}^{4p} \right] \le C t_k^{2p}.
  \end{equation*}
\end{proposition}
\begin{proof}
  The proof heavily relies on \emph{Burkholder's inequality},
  see~\cite{burk73}. Recall that the discrete time Burkholder inequality
  establishes the existence of constants $c_p, C_p$ for $1 < p < \infty$ such
  that for every $p$-integrable real martingale $Y_n$ and any $n\in\N$ we have
  \begin{equation*}
    c_p \sup_n \norm{S_n}_{L^p} \le \sup_n \norm{Y_n} \le C_p \sup_n
    \norm{S_n}_{L^p}
  \end{equation*}
  where, setting $Y_0 \coloneqq 0$, $S_n \coloneqq \sqrt{\sum_{k=1}^n (Y_k -
    Y_{k-1})^2}$ is the square root of the quadratic variation of $Y$. By
  choosing $Y_{n+l} \equiv Y_n$ for $l > 0$, this immediately implies the
  corresponding finite version
  \begin{equation}
    \label{eq:burkholder}
    c_p \norm{S_n}_{L^p} \le \norm{Y_n}_{L^p} \le C_p \norm{S_n}_{L^p}.
  \end{equation}
  By equivalence of homogeneous norms, see, for
  instance,~\cite[Theorem~7.44]{fri/vic10}, we can replace the
  Carnot-Caratheodory norm $\norm{\cdot}$ on $G^2(\R^d)$ by the homogeneous
  norm
  \begin{equation*}
    \norm{\mathbf{x}}_2 \coloneqq \max\left( \abs{\pi_1(\log(\mathbf{x}))},
      \sqrt{\abs{\pi_2(\log(\mathbf{x}))}} \right) \le
    \abs{\pi_1(\log(\mathbf{x}))} + \sqrt{\abs{\pi_2(\log(\mathbf{x}))}},
    \quad \mathbf{x} \in G^2(\R^d), 
  \end{equation*}
  where $\pi_1$ and $\pi_2$ denote the projection to the first and second
  level components of $\mathbf{x}$, i.e., when $\mathbf{x} = 1 + x + a \in
  G^2(\R^d)$, then $\pi_1(\mathbf{x}) = x \in \R^d$ and $\pi_2(\mathbf{x}) = a
  \in \R^d \otimes \R^d$. Thus, the assertion of the proposition is
  equivalent to the existence of a constant $C$ (only depending on $p$) such
  that
  \begin{align}
    \label{eq:1}
    E\left[ \abs{\pi_1(\log(\Xi^n_k))}^{4p} \right] &\le C t_k^{2p},\\
    \label{eq:2}
    E\left[ \abs{\pi_2(\log(\Xi^n_k))}^{2p} \right] &\le C t_k^{2p}.
  \end{align}

  We start by proving~(\ref{eq:1}). By the Campbell-Baker-Hausdorff formula,
  we have
  \begin{equation*}
    \abs{\pi_1(\log(\Xi^n_k))} = \abs{\sum_{i=1}^d \left(\sum_{l=1}^k X^i_l
      \right)e_i} \le C \sum_{i=1}^d \abs{\sum_{l=1}^k X^i_l}.
  \end{equation*}
  Here, $X^i_l = \sqrt{\dt_l} X^i_{(l)}$ and $C$ is a constant, which does
  neither depend on the partition $\mathcal{D}_n$ nor on $k$. For the
  remainder of the proof, we will use this symbol for constants that may vary
  from line to line, but do not depend on $\mathcal{D}_n$ or on $k$. This
  implies that
  \begin{equation*}
    E\left[ \abs{\pi_1(\log(\Xi^n_k))}^{4p} \right] \le C \sum_{i=1}^d E\left[
      \abs{\sum_{l=1}^k X^i_l}^{4p} \right].
  \end{equation*}
  Now we apply Burkholder's inequality to the martingale $Y_k = \sum_{l=1}^k
  X^i_l$ with the exponent $4p$ to get
  \begin{equation*}
    E\left[ \abs{\pi_1(\log(\Xi^n_k))}^{4p} \right] \le C \sum_{i=1}^d E\left[
      \abs{\sum_{l=1}^k (X^i_l)^2}^{2p} \right] = C t_k^{2p} \sum_{i=1}^d
    E\left[ \abs{ \sum_{l=1}^k \f{\dt_l}{t_k} (X^i_{(l)})^2}^{2p} \right].
  \end{equation*}
  Noting that the sum inside the expectation is a convex combination, we apply
  Jensen's inequality for the convex function $x^{2p}$ and get
  \begin{equation*}
    E\left[ \abs{\pi_1(\log(\Xi^n_k))}^{4p} \right] \le C t_k^{2p}
    \sum_{i=1}^d \sum_{l=1}^k \f{\dt_l}{t_k} E\left[(X^i_{(l)})^{4p} \right] =
    C \left( \sum_{i=1}^d E\left[(X^i_{(l)})^{4p} \right] \right) t_k^{2p},
  \end{equation*}
  which is of the form required in~(\ref{eq:1}).

  For~(\ref{eq:2}), we again start with the Campbell-Baker-Hausdorff formula
  and get
  \begin{align}
    E\left[ \abs{\pi_2(\log(\Xi^n_k))}^{2p} \right] &= E\left[ \abs{
          \sum_{1 \le i < j \le d} \left[ \sum_{l=1}^k A^{i,j}_l + \f{1}{2}
            \sum_{1 \le l_1 < l_2 \le k} \left( X^i_{l_1} X^j_{l_2} -
              X^j_{l_1} X^i_{l_2} \right) \right] [e_i,e_j]}^{2p}
      \right]\nonumber \\
      &\le C \sum_{1\le i<j\le d} E\left[ \abs{ \sum_{l=1}^k A^{i,j}_l + \f{1}{2}
          \sum_{1 \le l_1 < l_2 \le k} \left( X^i_{l_1} X^j_{l_2} -
            X^j_{l_1} X^i_{l_2} \right) }^{2p} \right] \nonumber \\
      &\le C \sum_{1\le i<j\le d} \left( E\left[ \abs{\sum_{l=1}^k
            A^{i,j}_l}^{2p}\right]  + E\left[ \abs{\sum_{1 \le l_1 < l_2 \le k}
            X^i_{l_1} X^j_{l_2}}^{2p} \right] + E\left[ \abs{\sum_{1 \le l_1 <
              l_2 \le k} X^j_{l_1} X^i_{l_2}}^{2p} \right]\right).\label{eq:5}
  \end{align}
  Now fix some $1 \le i < j \le d$. Again by Jensen's inequality, we have
  \begin{equation}
    \label{eq:3}
    E\left[ \abs{\sum_{l=1}^k A^{i,j}_l}^{2p}\right] = t_k^{2p} E\left[
      \abs{\sum_{l=1}^k \f{\dt_l}{t_k} A^{i,j}_{(l)}}^{2p}\right] \le t_k^{2p}
    \sum_{l=1}^k \f{\dt_l}{t_k} E\left[ \abs{A^{i,j}_{(l)}}^{2p} \right] =
    E\left[ \abs{A^{i,j}_{(1)}}^{2p} \right] t_k^{2p}.
  \end{equation}
  On the other hand, note that
  \begin{equation*}
    \sum_{1 \le l_1 < l_2 \le k} X^i_{l_1} X^j_{l_2} = \sum_{l_2 = 1}^k
    X^j_{l_2} \left( \sum_{l_1 = 1}^{l_2 - 1} X^i_{l_1} \right)
  \end{equation*}
  is a martingale (indexed by $k$). Thus, Burkholder's
  inequality~(\ref{eq:burkholder}) for the exponent $2p$ gives
  \begin{equation*}
    E\left[ \abs{\sum_{1 \le l_1 < l_2 \le k} X^i_{l_1} X^j_{l_2}}^{2p}
    \right] \le C E\left[ \abs{ \sum_{l_2 = 1}^k (X^j_{l_2})^2 \left(
          \sum_{l_1=1}^{l_2-1} X^i_{l_1} \right)^2 }^p \right] = C t_k^p
    E\left[ \abs{ \sum_{l_2 = 1}^k \f{\dt_{l_2}}{t_k} (X^j_{(l_2)})^2 \left(
          \sum_{l_1=1}^{l_2-1} X^i_{l_1} \right)^2 }^p \right].
  \end{equation*}
  Now we again apply Jensen's inequality and then the Cauchy-Schwarz
  inequality, and obtain
  \begin{align}
    E\left[ \abs{\sum_{1 \le l_1 < l_2 \le k} X^i_{l_1} X^j_{l_2}}^{2p} 
    \right] &\le C t_k^p \sum_{l_2 = 1}^k \f{\dt_{l_2}}{t_k} E\left[\abs{
        (X^j_{(l_2)})^2 \left(\sum_{l_1=1}^{l_2-1} X^i_{l_1} \right)^2}^p
    \right]\nonumber \\
    &\le C t_k^p \sum_{l_2 = 1}^k \f{\dt_{l_2}}{t_k} \left(
      E\left[\abs{X^j_{(l_2)}}^{4p} \right] \right)^{1/2} \left( E\left[
        \abs{\sum_{l_1=1}^{l_2-1} X^i_{l_1}}^{4p} \right]
    \right)^{1/2}.\label{eq:4} 
  \end{align}
  By applying Burkholder's and Jensen's inequalities for a final time, we get
  for the left-most term in the above inequality
  \begin{multline*}
    E\left[ \abs{\sum_{l_1=1}^{l_2-1} X^i_{l_1}}^{4p} \right] \le C
    t_{l_2-1}^{2p} E\left[ \abs{\sum_{l_1=1}^{l_2-1} \f{\dt_{l_1}}{t_{l_2-1}}
        (X^i_{(l_1)})^2}^{2p} \right] \\ \le C t_{l_2-1}^{2p}
    \sum_{l_1=1}^{l_2-1} \f{\dt_{l_1}}{t_{l_2-1}} E\left[
      \abs{X^i_{(l_1)}}^{4p} \right] \le C t_{l_2-1}^{2p} E\left[
      \abs{X^i_{(1)}}^{4p} \right].
  \end{multline*}
  Inserting the last inequality into~(\ref{eq:4}), we obtain
  \begin{align*}
    E\left[ \abs{\sum_{1 \le l_1 < l_2 \le k} X^i_{l_1} X^j_{l_2}}^{2p} 
    \right] &\le C t_k^p \sum_{l_2=1}^k \f{\dt_{l_2}}{t_k} t_{l_2-1}^p \left(
      E\left[\abs{X^j_{(1)}}^{4p} \right] \right)^{1/2} \left(
      E\left[\abs{X^i_{(1)}}^{4p} \right] \right)^{1/2} \\
    &\le C \left(
      E\left[\abs{X^j_{(1)}}^{4p} \right] \right)^{1/2} \left(
      E\left[\abs{X^i_{(1)}}^{4p} \right] \right)^{1/2} t_k^{2p}.
  \end{align*}
  Together with~(\ref{eq:5}) and~(\ref{eq:3}) this shows~(\ref{eq:2}),
  and the proposition follows.
\end{proof}


\section{Proof of the main results}
\label{sec:proof-main-result}

Analogously to~\cite[Theorem 1]{bre/fri/hue09} we can now state our
\begin{theorem}
  \label{thr:donsker-for-geodesic-rws}
  Let $\mathcal{D}_n = \{0 = t_0 < \cdots < t_n = 1\}$ be a sequence of meshes
  with $\abs{\mathcal{D}_n} \to 0$ and let $\Xi^n =
  \left(\Xi^n_k\right)_{k=0}^n$ be a centered random walk in $G^2(\R^d)$ along
  the mesh (as defined in Section~\ref{sec:random-walks-with}), whose
  increments have moments of all orders. Additionally, we impose
  $E(\pi_2(\xi^n_k)) = 0$.  Define a sequence $\barXi^n$ of stochastic
  processes with values in $G^2(\R^d)$ by $\barXi^n_{t_k} = \Xi^n_k$ for $k =
  0, \ldots, n$ and by \emph{geodesic} interpolation for $t \in
  [t_k,t_{k+1}]$. Then
  \begin{equation*}
    \barXi^n \xrightarrow[n\to\infty]{} \fB
  \end{equation*}
  in $C^{0,\alpha-\text{H\"{o}l}}\left([0,1],G^2(\R^d)\right)$ for all
  $\alpha<1/2$.
\end{theorem}
\begin{proof}
  Kolmogorov's tightness criterion, see, for instance,~\cite{rev/yor99},
  implies that $\barXi^n$ is tight (in $C^{0,\alpha-\text{H\"ol}}$) provided
  that for any $u,v \in [0,1]$
  \begin{equation}
    \label{eq:kolmogorov}
    \sup_n E\left[ d\left( \barXi^n_v, \barXi^n_u \right)^a \right] \le c
    \abs{v-u}^{1+b},
  \end{equation}
  where $a$, $b$, $c$ are positive constants with $\alpha < b/a$ and $d$
  denotes the Carnot-Caratheodory distance defined by $d(x,y) =
  \norm{x^{-1}y}$. We choose $a = 4p$ and $b=2p-1$, then
  Proposition~\ref{prop:2} implies that~(\ref{eq:kolmogorov}) holds for $u,v
  \in \mathcal{D}_n$. For arbitrary $v<u$, assume that $t_i \le v < t_{i+1}$
  and $t_j \le u < t_{j+1}$ (for $t_i$, $t_{i+1}$, $t_j$, $t_{j+1} \in
  \mathcal{D}_n$). Using (by the geodesic interpolation)
  \begin{equation}\label{eq:geodesic-interpolation}
    d\left( \barXi^n_v,\barXi^n_{t_{i+1}} \right) = \f{t_{i+1} - v}{\dt_{i+1}}
    d\left( \barXi^n_{t_{i}}, \barXi^n_{t_{i+1}} \right), \quad d\left(
      \barXi^n_{t_j},\barXi^n_u \right) = \f{u - t_j}{\dt_{j+1}} d\left(
      \barXi^n_{t_j}, \barXi^n_{t_{j+1}} \right)
  \end{equation}
  and the triangle inequality, we obtain
  \begin{align*}
    E\left[ d\left( \barXi^n_v, \barXi^n_u \right)^a \right] &\le
    \tilde{c} \left[ \left( \f{t_{i+1} - v}{\dt_{i+1}} \right)^{2p}
      (t_{i+1}-v)^{2p} + (t_j - t_{i+1})^{2p} + \left(\f{u - t_j}{\dt_{j+1}}
      \right)^{2p} (u - t_j)^{2p} \right] \\
    &\le c \abs{u-v}^{2p},
  \end{align*}
  for some constant $c$ only depending on $p$. 

  This shows that the sequence of stochastic processes $\barXi^n$ is tight in
  rough-path-topology. Moreover, Lemma~\ref{lem:clt-non-iid} shows that the
  finite-dimensional marginal distributions of $\barXi^n$ converge to those of
  $\fB$. Thus, we obtain the theorem for $\alpha < \f{4p}{2p-1}$ and, with $p
  \to \infty$, for any $\alpha < 1/2$.
\end{proof}

\begin{proof}[Proof of Theorem~\ref{thr:1}]
  In Theorem~\ref{thr:donsker-for-geodesic-rws} above, we have already proved
  our main result for $N = 2$ and for the special case of geodesic
  interpolation, i.e., for the case that $S_2\left( W^{\mathcal{D}_n}
  \right)_{0,t}$, $t \in [0,1]$, provides a geodesic interpolation between the
  grid points, i.e., between $S_2\left( W^{\mathcal{D}_n}\right)_{0,t_k}$,
  $t_k \in \mathcal{D}_n$. We note that the extension of the result to $N > 2$
  is immediate since the Lyons lift $S_N: G^2(\R^d) \to G^N(\R^d)$ is
  continuous in rough path topology, see, for instance,~\cite[Corollary
  9.11]{fri/vic10}.

  We need to give a proof for the case $N=2$ but without geodesic
  interpolation. As before, assume that we are given $t_i \le v < t_{i+1}$,
  $t_j \le u < t_{j+1}$, with $t_i, t_j, t_{i+1}, t_{j+1} \in
  \mathcal{D}_n$. For any path $\omega \in \mathcal{H}$, the Cameron-Martin
  space, we have, see~\cite[Proposition 15.7]{fri/vic10},
  \begin{equation*}
    \abs{\omega}_{1-\text{var};[s,t]} \le \sqrt{\abs{t-s}}
    \norm{\omega}_{\mathcal{H};[s,t]},
  \end{equation*}
  where $\abs{\cdot}_{1-\text{var};[s,t]}$ denotes the first variation of a
  path restricted to $[s,t]$ and $\norm{\cdot}_{\mathcal{H};[s,t]}$ denotes
  the Cameron-Martin norm, likewise restricted to $[s,t]$, i.e.,
  \begin{equation*}
    \norm{\omega}^2_{\mathcal{H};[s,t]} = \int_s^t \abs{\dot{\omega}(u)}^2 du.
  \end{equation*}
  Notice that
  \begin{equation*}
    \norm{W^{\mathcal{D}_n}}^2_{\mathcal{H};[t_j,u]} = \int_{t_j}^u
    \abs{\dot{W^{\mathcal{D}_n}_t}}^2 dt = \f{1}{\dt_{j+1}} \int^u_{t_j}
    \abs{\dot{W}\left( \f{t-t_j}{\dt_{j+1}} \right)}^2 dt
    \le \int_0^1 \abs{\dot{W}_t}^2 dt = \norm{W}^2_{\mathcal{H}}.
  \end{equation*}
  Therefore, we can bound
  \begin{align*}
    d\left( S_2\left( W^{\mathcal{D}_n} \right)_{0,t_j}, S_2\left(
        W^{\mathcal{D}_n} \right)_{0,u} \right) &= \norm{S_2\left(
        W^{\mathcal{D}_n} \right)_{t_j,u}} \\
    &\le \abs{W^{\mathcal{D}_n}}_{1-\text{var};[t_j,u]} \\
    &\le \sqrt{u-t_j} \norm{W}_{\mathcal{H}}.
  \end{align*}
  By Assumption~\ref{ass:interpolation}, we get
  \begin{equation}
    \label{eq:7}
    E\left[ d\left( S_2\left( W^{\mathcal{D}_n} \right)_{0,t_j}, S_2\left(
          W^{\mathcal{D}_n} \right)_{0,u} \right)^{4p} \right] \le
    (u-t_j)^{2p} E\left[ \norm{W}^{4p}_{\mathcal{H}} \right] \le C (u-t_j)^{2p},
  \end{equation}
  and similarly
  \begin{equation}
    \label{eq:8}
        E\left[ d\left( S_2\left( W^{\mathcal{D}_n} \right)_{0,v}, S_2\left(
          W^{\mathcal{D}_n} \right)_{0,t_{i+1}} \right)^{4p} \right] \le C
    (t_{i+1} - v)^{2p}.
  \end{equation}
  Therefore, we can show tightness of the sequence of processes $S_2\left(
    W^{\mathcal{D}_n} \right)_{0,\cdot}$ as in the proof of
  Theorem~\ref{thr:donsker-for-geodesic-rws} above.
\end{proof}

Now we finally return to the Donsker theorem for cubature paths with an
adjourned, Lipschitz component $h$.
\begin{proof}[Proof of Corollary~\ref{cor:donsker-with-drift}.]
  Let $h^{\mathcal{D}_n} \coloneqq \Prefix^{h}{W}^{\mathcal{D}_n,0}$ denote
  the $0$-component of $\Prefix^{h}{W}^{\mathcal{D}_n}$. Moreover, let
  $\id:[0,1] \to [0,1]$ denote the identity, $\id(t) = t$. We note that
  $h^{\mathcal{D}_n}$ converges to $\id$ in
  $C^{0,\beta-\text{H\"{o}l}}([0,1];\R)$ for any $\beta < 1$. Indeed, let $0 <
  t < 1$ and let $t_i \le t < t_{i+1}$ denote the grid points closest to
  $t$. Then apparently
  \begin{equation*}
    h^{\mathcal{D}_n}(t) = t_i + \dt_{i+1} h\left( \f{t-t_i}{\dt_{i+1}}
    \right).
  \end{equation*}
  For $h(0) = 0$, we get $\abs{h(t)} \le L \abs{t}$, where $L$ is the
  Lipschitz constant of $h$.  From this one can easily conclude that
  $\norm{h^{\mathcal{D}_n} - \id}_\infty \le (1+L) \abs{\mathcal{D}_n}$ and
  $h^{\mathcal{D}_n}$ converges to $\id$ uniformly on $[0,1]$ with uniform
  Lipschitz bounds for $\abs{\mathcal{D}_n} \to 0$. This implies the
  convergence in $\beta$-H\"{o}lder topology for any $\beta < 1$.

  By this result together with Theorem~\ref{thr:1} for the convergence of
  $S_2(W)$, we can immediately conclude that $S_2\left(
    \Prefix^{h}{W}^{\mathcal{D}_n} \right)_{0,\cdot}$ converges to
  $S_2(B)_{0,\cdot}$ in $C^{0,\alpha-\text{H\"{o}l}}\left( [0,1]; G^2_1(\R^d)
  \right)$. By continuity of the Lyons lift $S_2(x) \mapsto S_N(x)$ in rough
  path topology, the statement of the corollary follows.
\end{proof}

\bibliographystyle{plain}
\bibliography{donsker.bib}

\end{document}